\documentclass{amsart}

\usepackage{amsmath} 
\usepackage{amssymb} 
\usepackage{amsthm} 
\usepackage[alphabetic]{amsrefs} 
\usepackage{colonequals} 
\usepackage{graphicx} 
\usepackage{xcolor} 

\newcommand{\C}{\mathbb{C}}
\newcommand{\F}{\mathbb{F}}
\newcommand{\N}{\mathbb{N}}
\renewcommand{\P}{\mathbb{P}}
\newcommand{\Q}{\mathbb{Q}}
\newcommand{\Z}{\mathbb{Z}}

\newcommand{\calE}{\mathcal{E}}
\newcommand{\calH}{\mathcal{H}}
\newcommand{\calM}{\mathcal{M}}
\newcommand{\calN}{\mathcal{N}}
\newcommand{\calQ}{\mathcal{Q}}

\newcommand{\Wtilde}{\widetilde{W}}

\DeclareMathOperator{\GL}{GL}
\renewcommand{\Im}{\mathrm{Im}}
\DeclareMathOperator{\lcm}{lcm}
\DeclareMathOperator{\new}{new}
\DeclareMathOperator{\old}{old}
\DeclareMathOperator{\SL}{SL}
\DeclareMathOperator{\Stab}{Stab}
\DeclareMathOperator{\Sym}{Sym}
\DeclareMathOperator{\Tr}{Tr}

\numberwithin{equation}{section}
\newtheorem{cor}[equation]{Corollary}
\newtheorem{lem}[equation]{Lemma}
\newtheorem{thm}[equation]{Theorem}

\theoremstyle{remark}

\newtheorem{remark}[equation]{Remark}
\newtheorem{exa}[equation]{Example}

\newcommand{\defi}[1]{\textsf{#1}}
\newcommand{\mat}[4]{\left(\begin{array}{cc} {#1} & {#2} \\ {#3} & {#4} \end{array} \right)}
\newcommand{\Mod}[1]{\ (\mathrm{mod}\ #1)}

\title{A note on the trace formula}
\author{Eran Assaf}
\address{Department of Mathematics, Dartmouth College, 6188 Kemeny Hall, Hanover, NH 03755, USA}
\email{eran.assaf@dartmouth.edu}

\begin{document}

\maketitle

\begin{abstract}
    In this mostly expository note, we prove explicit formulas for the traces of Hecke operators on spaces of cusp forms fixed by Atkin-Lehner involutions, which are suitable for efficient implementation. In addition, we correct a couple of errors in previously published formulas.
\end{abstract}

\section{Introduction}

Recently, a significant statistical phenomenon was discovered \cite{HLOP22} in the average values of the Dirichlet coefficients of elliptic curves. This phenomenon was further studied by \cite{Cowan23}, and in order to facilitate these investigations over large data sets, the need for an efficient implementation arose.

Specifically, for a given conductor $N$, let $\calE_N$ be the set of isogeny classes of elliptic curves defined over $\Q$ of conductor $N$. For an elliptic curve $E \in \calE_N$ and a prime $p$, let $a_p(E) = p + 1 - \#E(\F_p)$, and associate to $E$ its $L$-function $L_E(s)$, a holomorphic function defined for $\Re(s) > 3/2$, whose completion $\Lambda_E(s) = \pi^{-s} \Gamma \left( \frac{s}{2} \right) \Gamma \left( \frac{s+1}{2} \right) L_E(s)$ can be analytically continued, and satisfies a functional equation 
\begin{equation}
\Lambda_E(2-s) = w(E) \Lambda_E(s),
\end{equation}
where $w(E) = \pm 1$ is a sign, which depends only on $E$.
Let $\calE_N^i$ be the subset of (isogeny classes) of elliptic curves such that $w(E) = (-1)^r$ for $r = 0,1$.
The murmuration phenomenon then describes a significant difference in the behaviour of the averages
$$
f_r(p) = \frac{1}{\#\calE_N^r}\sum_{E \in \calE_N^r} a_p(E),
$$
when $r = 0$ as opposed to $r = 1$.

By the modularity theorem, each isogeny class of elliptic curves with conductor $N$ and $w(E) = (-1)^r$ corresponds to a certain cusp form of weight $2$ and level $N$, $f \in S_2(N)^{\new}$ with $T_p(f) = a_p(E) f$ and $W_N(f) = w(E) f$, where $T_p$ is the Hecke operator at $p$ and $W_N$ is the Fricke involution, as defined in Section~\ref{sec: modular forms}. It is therefore sensible to assume that this phenomenon is a shadow of a phenomenon occurring for modular forms.

Explicitly, using the notations of Section~\ref{sec: modular forms}, if we write $S_k(N)^{\new, +}$ for the subspace of cusp forms $f$ of weight $k$ and level $N$ satisfying $W_N(f) = f$, and $S_k(N)^{\new, -}$ for the subspace of cusp forms $f$ satisfying $W_N(f) = -f$, then we are interested in comparing the quantities
$$
\frac{1}{\# S_k(N)^{\new, +}} \Tr(T_p | S_k(N)^{\new, +}), \quad
\frac{1}{\# S_k(N)^{\new, -}} \Tr(T_p | S_k(N)^{\new, -}).
$$

We further note the relations 
$$
\Tr(T_p | S_k(N)^{\new, +}) + \Tr(T_p | S_k(N)^{\new, -}) = 
\Tr(T_p | S_k(N)^{\new}),
$$
and
$$
\Tr(T_p | S_k(N)^{\new, +}) - \Tr(T_p | S_k(N)^{\new, -}) = 
\Tr(T_p \circ W_N | S_k(N)^{\new}).
$$

While many efficient implementations of $\Tr(T_p | S_k(N)^{\new})$ exist, efficient implementation of $\Tr(T_p \circ W_N | S_k(N)^{\new})$ was needed. 

This note explains how to obtain formulas for these and related quantities, suitable for implementation, using the trace formula. It is based on the formula in \cite{P}*{Theorem 4}, and follows it closely. 
However, as \cite{P} does not describe explicitly a trace formula for the new subspace, we follow the methods of \cite{SZ} to obtain one.

\subsection{Results}

Before stating our main result, we introduce some notations. Let $n,N$ be positive integers, and let $Q \parallel N$ be a positive integer which exactly divides $N$, i.e. $Q \mid N$ and $(Q, N/Q) = 1$.

For integers $d', n'$ we write $d'_{N/Q} = \gcd(d, N/Q)$ and $n'_{N/Q} = \gcd(n', N/Q)$. For any $N' | N$, we write $Q' = (N',Q)$, and we let $\calQ$ be the set of primes dividing $Q$.

Let $\mu$ be the M{\"o}bius function, namely the multiplicative function with $\mu(p) = -1$ and $\mu(p^e) = 0$ for all $e \ge 2$ for all primes $p$, and denote by $\sigma_{0,n}(m)$ the number of divisors of $m$ that are prime to $n$.
The set $\calN_{Q,N,n}(d,n')$ is defined as in \eqref{eq: N(d,n) sets}, using \eqref{eq: divisors primr to Nprime} and \eqref{eq: square divisors}, and $\alpha_{\calQ,n}$ is the multiplicative function defined by \eqref{eq: alpha}. 

Our main result is the following formula (Corollary~\ref{cor: main thm}), relating the trace of $T_n \circ W_Q$ on the new subspace, where $T_n$ is the $n$-th Hecke operator and $W_Q$ is the Atkin-Lehner operator, to traces of $T_{n'} \circ W_{Q'}$ on cusp forms of lower levels, where $n', Q'$ range over numbers dividing $n, Q$ respectively.

\begin{thm} 
    Let $n, N$ be positive integers, and let $Q \parallel N$. 
    Denote 
    \begin{align*}
    T_{<n,k} &(Q,N) \\
    &= \sum_{\substack{d | n' \mid n \\ n'_{N/Q} = d_{N/Q}^2 \\ 1 < n'}}
        (n')^{\frac{k}{2}} 
        \frac{\mu(d)}{d}
        \sum_{N'}
        \sigma_{0,n} \left( 
        \frac{N/Q}{N'/Q'}
        \right)
        \Tr \left(T_{\frac{n}{n'}} \circ W_{Q'} \middle |  S_k(N')^{\new} \right),
    \end{align*}
    where $N'$ ranges over the set $\calN_{Q,N,n}(d, n')$ defined in \eqref{eq: N(d,n) sets}.
    Then
    \begin{align*}
        \Tr&(T_n \circ W_Q | S_k(N)^{\new}) \\
        &= \sum_{N' \mid N}
        \alpha_{\calQ,n} \left( \frac{N}{N'} \right)
        \left(
        \Tr \left(T_n \circ W_{Q'} \middle | S_k(N') \right)
        - T_{<n,k}(Q', N')
        \right).
    \end{align*}      
\end{thm}

Since \cite{P}*{Theorem 4} gives a formula for $\Tr(T_n \circ W_Q | S_k(N))$, combining it with our result provides a formula for the trace on the new subspace. For the reader's convenience we introduce Popa's notations and results, with \cite{P}*{Theorem 4} appearing as Theorem~\ref{thm: Popa trace formula}.

The main tool used in this work is an analysis of the commutation relations between Hecke operators and level-raising operators, which is described in Section~\ref{sec: level raising}. This allows one to relate the traces on a space of cusp forms with traces of newforms of lower levels, which is the content of Section~\ref{sec: trace formulas}.
The final ingredient is the M{\"o}bius inversion formula, allowing us to invert the relation and describe the trace on the space of newforms as a function of traces on spaces of cusp forms of lower levels.

The formula has been implemented in Magma~\cite{Magma} and verified for many possible inputs using the Modular Symbols package of Magma. The case $n = p$ and $Q = N$ (Corollary~\ref{cor: T_p W_N new}) has also been implemented in Pari/GP~\cite{Pari} for efficiency reasons, and has been used to construct the plots of murmurations in modular forms that are found in \cite{Drew}. The implementation is available at \cite{repo}.

In addition, while working on this project and implementing the formulas, errors were discovered in one of the formulas in \cite{SZ} and in one of the Lemmas in \cite{P}. Their resolution is described in Appendix~\ref{sec: fixing SZ}, and we present here the corrected version (Theorem~\ref{thm: SZ fix} and Lemma~\ref{lem: Popa fix}). 

\begin{thm}
Assume $(l,m) = 1$ and $n \mid m$ with $\left(n, \frac{m}{n} \right) = 1$. Let $T_{l}$ and $W_n$ denote the $l$-th Hecke operator and $n$-th Atkin-Lehner involution on $M_{2k-2}(m)$ as in \cite{SZ}. Then
    $$
    \Tr(T_{l} \circ W_n, S_{2k-2}(m)) = 
    \sum_{\substack{m' \mid m \\ m/m' \text{ squarefree}}}
    \mu \left( \frac{n}{(n,m')} \right) s_{k,m'}(l,(n,m')),
    $$
    where $s_{k,m}(l,n)$ is the quantity described in \cite{SZ}*{Theorem 1}.
\end{thm}

\subsection{Acknowledgements}
The author would like to thank Alex Cowan and Andrew  Sutherland for motivating this project, and Steve Fan, Alexandru Popa and Andrew Sutherland for helpful comments and discussions. The author was supported by a Simons Collaboration grant (550029, to Voight). 

\section{Modular Forms} \label{sec: modular forms}
In this section we review some definitions and set notations and conventions for the spaces of modular forms we will be working with. We are mostly following the notations in \cite{DS}, and the reader is invited to consult it for further reference.

The \defi{upper half plane} is 
\begin{equation}
\calH = \{ \tau \in \C : \Im(\tau) > 0 \}.
\end{equation}

Denote by $\GL_2^{+}(\Q)$ the group of $2 \times 2$ matrices with rational entries and positive determinant, namely
\begin{equation}
    \GL_2^+(\Q) = \left \{
    \mat abcd : \ a,b,c,d \in \Q, \ ad - bc > 0
    \right \}.
\end{equation}

Then $\GL_2^+(\Q)$ acts on $\calH$ via M{\"o}bius transformations, explicitly 
\begin{equation} \label{eq: mobius action}
    \alpha = \mat abcd : \tau \mapsto \alpha \tau = \frac{a \tau + b}{c \tau + d}.
\end{equation}

For any matrix $\alpha \in \GL_2^+(\Q)$ as above, and $\tau \in \calH$ we write $j(\alpha, \tau) = c \tau + d$ for the factor of automorphy.

For any integer $k$, the action \eqref{eq: mobius action} induces a right action on the space of holomorphic functions $f : \calH \to \C$, given by 
\begin{equation} \label{eq: GL2 weight-k action}
(f [\alpha]_k )(\tau) = \det(\alpha)^{k-1} j(\alpha, \tau)^{-k} f(\alpha \tau).
\end{equation}

The \defi{modular group} is the group $2 \times 2$ matrices with integral entries and determinant $1$, 
\begin{equation}
\SL_2(\Z) = \left \{
\mat abcd : \ a,b,c,d \in \Z, \ ad-bc = 1
\right \},
\end{equation}
and we denote by $\Gamma_0(N)$ its subgroup of matrices upper-triangular modulo $N$,
\begin{equation}
    \Gamma_0(N) = \left \{ 
    \mat abcd \in \SL_2(\Z) : \ c \in N \Z
    \right \}.
\end{equation}

If $f: \calH \to \C$ is holomorphic and $\Z$-periodic, it admits a Laurent expansion 
\begin{equation} \label{eq: q-expansion}
    f(\tau) = \sum_{n = n_0}^{\infty} a_n q^n, \quad q = e^{2 \pi i \tau}.
\end{equation}
We say $f$ \defi{is holomorphic} (resp. \defi{vanishes}) at $\infty$ if $n_0 \ge 0$ (resp. $n_0 \ge 1$).

A holomorphic function $f : \calH \to \C$ is a \defi{modular form of weight $k$ and level $N$} if $f[\gamma]_k = f$ for all $\gamma \in \Gamma_0(N)$ and $f[\alpha]_k$ is holomorphic at $\infty$ for all $\alpha \in \SL_2(\Z)$. 
If $f[\alpha]_k$ vanishes at $\infty$ for all $\alpha \in \SL_2(\Z)$, $f$ is a \defi{cusp form}.

Denote by $S_k(N)$ the space of cusp forms of weight $k$ and level $N$.

The space $S_k(N)$ admits an inner product $\langle , \rangle : S_k(N) \times S_k(N) \to \C$, the \defi{Petersson inner product}, given by
\begin{equation} \label{eq: Petersson inner product}
    \langle f, g \rangle 
    = \frac{1}{V_{N}} \int_{\Gamma_0(N) \backslash \calH}
    f(\tau) \overline{g(\tau)} \Im(\tau)^k d \mu(\tau),
\end{equation}
where 
\begin{equation} \label{eq: hyperbolic measure}
d \mu (\tau) = \frac{dx dy}{y^2}, \quad \tau = x + iy \in \calH
\end{equation}
is the \defi{hyperbolic measure} on the upper half plane, and $V_N = \int_{\Gamma_0(N) \backslash \calH} d \mu(\tau)$.

Let $d \mid N$ be a divisor of $N$. Then the element 
\begin{equation} \label{eq: level raising}
    \alpha_d = \mat d 0 0 1
\end{equation}
gives rise to an injective linear map $[\alpha_d]_k : S_k(Nd^{-1}) \to S_k(N)$, hence to a map $i_d : S_k(Nd^{-1})^2 \to S_k(N)$ given by
\begin{equation} \label{eq: level raising total}
    i_d(f,g) = f + g[\alpha_d]_k.
\end{equation}

When the weight $k$ is clear from the context, we will abuse notation and write also $\alpha_d : S_k(Nd^{-1}) \to S_k(N)$ for $[\alpha_d]_k$.

The subspace of \defi{oldforms at level $N$} is 
\begin{equation} \label{eq: old subspace}
    S_k(N)^{\old} = \sum_{\substack{p \mid N \\ p \text{ prime}}} 
    i_p \left( S_k(Np^{-1})^2
    \right),
\end{equation}
and the subspace of \defi{newforms at level $N$} is the orthogonal complement with respect to the Petersson inner product,
\begin{equation} \label{eq: newspace}
    S_k(N)^{\new} = \left( S_K(N)^{\old} \right)^{\perp}.
\end{equation}


For any integer $n$, denote by
$$
\Delta_n = \left \{ \delta = \mat a b c d \in M_2(\Z) : \det \delta = n, c \in N \Z, (a, N) = 1 \right \}
$$
the set of $2 \times 2$ matrices with integral entries of determinant $n$ which are upper triangular modulo $N$. Let $T_n : S_k(N) \to S_k(N)$ be the $n$-th \defi{Hecke operator}, defined as  
\begin{equation} \label{eq: Hecke operator double coset}
  T_n(f) = \sum_{\beta \in \Gamma_0(N) \backslash \Delta_n} f [\beta]_k.
\end{equation}

For an integer $Q \parallel N$ (i.e. $Q \mid N$ and $(Q, \frac{N}{Q}) = 1$), we let $\Wtilde_Q : S_k(N) \to S_k(N)$ be the $Q$-th Atkin-Lehner operator, defined as 
\begin{equation} \label{eq: Atkin-Lehner operator}
    \Wtilde_Q(f) = f \left[ \mat {Qx}{y}{Nz}{Qw} \right]_k, \quad x,y,z,w \in \Z, \ xwQ - yz(N/Q) = 1,
\end{equation}
and let $W_Q = Q^{1-\frac{k}{2}} \Wtilde_Q$ be the Atkin-Lehner involution.

\section{Level raising and Hecke operators} \label{sec: level raising}

Since we are interested in computing traces of linear operators on the space $S_k(N)^{\new}$, we need to understand how these operators behave with respect to the level-raising operators.
We begin with the Atkin-Lehner operator, as this is quite straight-forward.

\begin{lem} \label{lem: AL and level-raising}
    Let $N' \mid N$, let $d \mid \frac{N}{N'}$ and let $Q \parallel N$. 
    Write $Q' = (Q, N')$ and $d' = \left(\frac{N}{dN'}, Q \right) \cdot \left( d, \frac{N}{Q} \right)$.
    We have 
    \begin{equation} \label{eq: AL with level raising}
    d^{1-\frac{k}{2}} \cdot W_Q \circ \alpha_d = 
    (d')^{1-\frac{k}{2}} \cdot \alpha_{d'} \circ W_{Q'}
    \end{equation}
\end{lem}

\begin{proof}
    First, let us note that as both $N' \mid N$ and $Q \mid N$, we have 
    $$
    \frac{QN'}{Q'} = \frac{QN'}{(Q,N')} =  \lcm(Q,N') \mid N,
    $$ 
    hence $\frac{N'}{Q'} \mid \frac{N}{Q}$ so that $(Q',\frac{N'}{Q'}) \mid (Q, \frac{N}{Q}) = 1$, establishing the existence of $W_{Q'}$. Since $d' \mid \frac{N}{dN'} \cdot d = \frac{N}{N'}$, $\alpha_{d'}$ is also well-defined.
    
    Let $x,y,z,w \in \Z$ be such that $xwQ - yz\frac{N}{Q} = 1$.
    Let $f \in S_k(N')$. Then
\begin{align*}
    \Wtilde_Q \circ \alpha_d (f) 
    &= f 
    \left [
    \mat d 0 0 1 
    \right ]_k 
    \left [ 
    \mat {Qx}{y}{Nz}{Qw}
    \right ]_k 
    = f \left [
    \mat {dQx}{dy}{Nz}{Qw}
    \right ]_k \\
    &= (d,Q)^{k-2}f \left [
    \mat {\frac{dQ}{(d,Q)}x}{\frac{d}{(d,Q)}y}{\frac{N}{(d,Q)}z}{\frac{Q}{(d,Q)}w}
    \right ]_k.
\end{align*}
On the other hand, from $Q \parallel N$ it follows that
\begin{equation} \label{eq: d' and squares}
(d,Q)^2 d' Q' = \left( (d,Q)\left(d,\frac{N}{Q}\right) \right) \cdot \left( (d,Q)\left(\frac{N}{dN'}, Q \right) \right) = d Q,
\end{equation}
hence
$$
\mat {\frac{dQ}{(d,Q)}x}{\frac{d}{(d,Q)}y}{\frac{N}{(d,Q)}z}{\frac{Q}{(d,Q)}w}
= 
\mat {d'Q'(d,Q)x}{(d,\frac{N}{Q})y}{\frac{N}{(d,Q)}z}{Q'(Q, \frac{N}{dN'})w},
$$
leading to
\begin{align*}
\Wtilde_Q \circ \alpha_d (f) 
    &= (d,Q)^{k-2} f \left [ 
    \mat {Q'(d,Q)x}{(d,\frac{N}{Q})y}{\frac{N}{(d,Q)d'}z} {Q'(Q, \frac{N}{dN'})w}
    \right ]_k 
    \left [
    \mat {d'} 0 0 1 
    \right ]_k \\
    &= (d,Q)^{k-2} \cdot \alpha_{d'} \circ \Wtilde_{Q'} (f).
    \end{align*}
    Therefore, we obtain
    \begin{align*}
    d^{1-\frac{k}{2}} \cdot W_Q \circ \alpha_d
    &= (dQ)^{1-\frac{k}{2}} \cdot \Wtilde_Q \circ \alpha_d
    = \left(\frac{dQ}{(d,Q)^2}\right)^{1-\frac{k}{2}} \cdot \alpha_{d'} \circ \Wtilde_{Q'} \\
    &= \left(\frac{dQ}{Q'(d,Q)^2}\right)^{1-\frac{k}{2}} \cdot \alpha_{d'} \circ W_{Q'}
    = (d')^{1-\frac{k}{2}} \cdot \alpha_{d'} \circ W_{Q'}. \qedhere
    \end{align*}
\end{proof}

\begin{cor} \label{cor: Fricke and level-raising}
    Let $N' \mid N$, and let $d \mid \frac{N}{N'}$. Write $d' = \frac{N}{N'd}$. 
    Then 
    \begin{equation} \label{eq: Fricke with level raising}
    d^{1-\frac{k}{2}} \cdot W_N \circ \alpha_d = 
    (d')^{1-\frac{k}{2}} \cdot \alpha_{d'} \circ W_{N'}
    \end{equation}
\end{cor}

\begin{remark}
    This also appears in \cite{AL78}*{Prop. 1.5}. We include it here since our normalization differs. The difference is due to the fact that $W_N$ is normalized with a power of $k/2$ and $\alpha_d$ is normalized with a power of $k-1$, which is convenient for the commutation relations with the Hecke operators, see below.
\end{remark}

We proceed to describe the relation between level-raising and the Hecke operators. For that, we need first to introduce a standard result, which can be obtained by considering explicit coset representatives.

\begin{lem}[\cite{DS}*{Ex.5.6.3}] \label{lem: T_p alpha_p base case}
    Let $N' \mid N$ and let $p$ be a prime.
    \begin{enumerate}
        \item If $p \nmid \frac{N}{N'}$ and $d \mid \frac{N}{N'}$, then $T_p \circ \alpha_d = \alpha_d \circ T_p$.
        \item If $N = N'p$, then $T_p \circ \alpha_p = p^{k-1} \alpha_1$.
        \item If $N = N'p$, then $T_p \circ \alpha_1 = \alpha_1 \circ T_p -  \delta_{p \nmid N'} \cdot \alpha_p$.
    \end{enumerate}
\end{lem}

Building on these base cases, we are ready to formulate the commutation relation.

\begin{lem} \label{lem: Hecke operator and level raising at primes}
    Let $p$ be a prime such that $p \mid \frac{N}{N'}$, and let $d \mid \frac{N}{N'}$. Then
    \begin{equation} \label{eq: Hecke with level raising}
        T_p \circ \alpha_d = \begin{cases}
            p^{k-1} \alpha_{\frac{d}{p}} & p \mid d \\
            \alpha_d \circ T_p - \delta_{p \nmid N'} \cdot \alpha_{pd}  & p \nmid d
        \end{cases}
    \end{equation}
\end{lem}

\begin{proof}
    If $p \mid d$, notice that $\alpha_d = \alpha_p \circ \alpha_{\frac{d}{p}}$. By Lemma~\ref{lem: T_p alpha_p base case} (2), we see that 
    $$
    T_p \circ \alpha_d = (T_p \circ \alpha_p) \circ \alpha_{\frac{d}{p}} 
    = p^{k-1} \alpha_1 \circ \alpha_{\frac{d}{p}} = p^{k-1} \alpha_{\frac{d}{p}}.
    $$
    If $p \nmid d$, let $e = v_p(N) - v_p(N')$, so that $\alpha_d = \prod_{i=1}^e \alpha_{1,i} \circ \alpha_{d,0}$, where the maps $\alpha_{1,i} : S_k(N/p^i) \to S_k(N/p^{i-1})$  and $\alpha_{d,0} : S_k(N') \to S_k(N/p^e)$ are intermediate level raising operators.

    Then $p \mid (N/p^i)$ for all $i \le e-1$, hence by Lemma~\ref{lem: T_p alpha_p base case} (3), $T_p \circ \alpha_{1,i} = \alpha_{1,i} \circ T_p$ for all $i \le e-1$, and $T_p \circ \alpha_{1,e} = \alpha_{1,e} \circ T_p - \delta_{p \nmid N'} \cdot \alpha_p$.
    Since Lemma~\ref{lem: T_p alpha_p base case} (1) yields $T_p \circ \alpha_{d,0} = \alpha_{d,0} \circ T_p$, it follows that
    \begin{align*}
    T_p \circ \alpha_d 
    &= \prod_{i=1}^{e-1} \alpha_{1,i} \circ (\alpha_{1,e} \circ T_p - \delta_{p \nmid N'} \cdot \alpha_p  ) \circ \alpha_{d,0} \\
    &= \prod_{i=1}^{e} \alpha_{1,i} \circ \alpha_{d,0} \circ T_p -  \delta_{p \nmid N'} \cdot \prod_{i=1}^{e-1} \alpha_{1,i} \circ \alpha_p \circ \alpha_{d,0}
    = \alpha_d \circ T_p - \delta_{p \nmid N'} \cdot \alpha_{pd} ,
    \end{align*}
    as claimed.
\end{proof}

As a corollary, we obtain a commutation relation for prime powers.

\begin{cor} \label{cor: Hecke operators and level raising prime powers}
    Let $p$ be a prime such that $p \mid \frac{N}{N'}$, and let $d \mid \frac{N}{N'}$.
    Let $r \ge 1$, and set $v_r = \min(v_p(d), r)$ and $T_{p^{-1}} = 0$. Then
    \begin{equation}
        T_{p^r} \circ \alpha_d = p^{v_r(k-1)} \left( \alpha_{\frac{d}{p^{v_r}}} \circ T_{p^{r-v_r}} - 
        \delta_{p \nmid N'} \cdot \alpha_{p \cdot \frac{d}{p^{v_r}}}  \circ T_{p^{r-v_r-1}} \right).
    \end{equation}
\end{cor}

\begin{proof}
    By induction on $r$. For $r = 1$, the statement is precisely Lemma~\ref{lem: Hecke operator and level raising at primes}.
    
    Since $p \mid N$, we have $T_{p^r} = T_p T_{p^{r-1}}$ on $S_k(N)$, hence by the induction hypothesis
    \begin{align} \label{eq: hecke operator prime powers induction step}
    \nonumber
    T_{p^r} \circ \alpha_d &= T_p \circ (T_{p^{r-1}} \circ \alpha_d) \\
    &= p^{v_{r-1}(k-1)} T_p \circ \left( \alpha_{\frac{d}{p^{v_{r-1}}}} \circ T_{p^{r-1-v_{r-1}}} - 
        \delta_{p \nmid N'} \cdot \alpha_{p \cdot \frac{d}{p^{v_{r-1}}}}  \circ T_{p^{r-v_{r-1}-2}} \right).
    \end{align}
    If $v_p(d) \le r - 1$, then $v_r = v_p(d) = v_{r-1}$, and writing $v \colonequals v_r$, \eqref{eq: hecke operator prime powers induction step} turns into
    $$
    T_{p^r} \circ \alpha_d = p^{v(k-1)} \left( \left( T_p \circ  \alpha_{\frac{d}{p^v}} \right) \circ T_{p^{r-1-v}} - 
        \delta_{p \nmid N'} \cdot \left( T_p \circ \alpha_{p \cdot \frac{d}{p^v}} \right)  \circ T_{p^{r-v-2}} \right).
    $$
    In such a case, since $p \nmid \frac{d}{p^v}$, Lemma~\ref{lem: Hecke operator and level raising at primes} yields
    \begin{align*}
    T_{p^r} &\circ \alpha_d \\
    &= p^{v(k-1)} \left( \left(\alpha_{\frac{d}{p^{v}}} \circ T_p - \delta_{p \nmid N'} \cdot \alpha_{p \cdot \frac{d}{p^{v}}}  \right) \circ T_{p^{r-1-v}} - 
        \delta_{p \nmid N'} \cdot p^{k-1} \alpha_{\frac{d}{p^{v}}}  \circ T_{p^{r-v-2}} \right) \\
    &= p^{v(k-1)} \left( \alpha_{\frac{d}{p^v}} \circ \left( T_p \circ T_{p^{r-v-1}} - \delta_{p\nmid N'} \cdot p^{k-1} T_{p^{r-v-2}} \right) - \delta_{p \nmid N'} \cdot \alpha_{p \cdot \frac{d}{p^v}} \circ T_{p^{r-v-1}} \right) \\
    &= p^{v(k-1)} \left( \alpha_{\frac{d}{p^v}} \circ T_{p^{r-v}} - \delta_{p \nmid N'} \cdot \alpha_{p \cdot \frac{d}{p^v}} \circ T_{p^{r-v-1}} \right),
    \end{align*}
    where the last equality follows from \cite{Lang}*{Thm VII.2.1, Thm VII.4.1}.

    On the other hand, if $v_p(d) \ge r$ then $v_{r-1} = r - 1$ and $p \mid \frac{d}{p^{r-1}}$. Combining \eqref{eq: hecke operator prime powers induction step} with Lemma~\ref{lem: Hecke operator and level raising at primes} we obtain
    \begin{equation*}
        T_{p^r} \circ \alpha_d
        =  p^{(r-1)(k-1)} T_p \circ \alpha_{\frac{d}{p^{r-1}}}
        = p^{r(k-1)} \alpha_{\frac{d}{p^r}}. \qedhere
    \end{equation*}
\end{proof}

Using the multiplicativity of Hecke operators, we get a general relation.

\begin{cor} \label{cor: Hecke and level raising general}
    Let $n \ge 1$ be an integer, and let $d \mid \frac{N}{N'}$. Then
    $$
    T_n \circ \alpha_d = (d, n)^{k-1} \sum_{\substack{d' \mid \left( \frac{n}{(d,n)}, \frac{N}{N'} \right) \\ (d', N') = 1}} \mu(d') \cdot \alpha_{\frac{dd'}{(d,n)}} \circ T_{\frac{n}{d'(d,n)}}.
    $$
\end{cor}

\begin{proof}
    By induction on the number of prime divisors of $n$.
    If $n = p^r$ this is Corollary~\ref{cor: Hecke operators and level raising prime powers}.
    For the induction step, let $m,n$ be integers with $(m,n) = 1$. Then 
    \begin{equation} \label{eq: Hecke operators and level raising multiplicative}
    T_{mn} \circ \alpha_d = T_m \circ (T_n \circ \alpha_d) 
    = (d,n)^{k-1} T_m \circ  \sum_{\substack{d_n \mid \left( \frac{n}{(d,n)}, \frac{N}{N'} \right) \\ 
    (d_n, N') = 1}} \mu(d_n) \cdot \alpha_{\frac{dd_n}{(d,n)}} \circ T_{\frac{n}{d_n(d,n)}}.
    \end{equation}
    Since $(m,n) = 1$, we have $(m, (d,n)) = 1$ and $(m,d_n) = 1$ for any $d_n \mid (n, \frac{N}{N'})$. It follows that for any $d_n \mid \left( \frac{n}{(d,n)}, \frac{N}{N'} \right)$ we have
    $$
    T_m \circ \alpha_{\frac{dd_n}{(d,n)}} = (d, m)^{k-1} \sum_{\substack{d_m \mid \left( \frac{m}{(d,m)}, \frac{N}{N'} \right) \\ (d_m, N') = 1}} \mu(d_m) \cdot \alpha_{\frac{dd_md_n}{(d,m)(d,n)}} \circ T_{\frac{m}{d_m(d,m)}}.
    $$
    Substituting into \eqref{eq: Hecke operators and level raising multiplicative}, we get
    \begin{equation}
    T_{mn} \circ \alpha_d = (d,mn)^{k-1} \sum_{\substack{d' \mid (\frac{mn}{(d,mn)}, \frac{N}{N'}) \\ (d', N') = 1}} \mu(d') \cdot \alpha_{\frac{dd'}{(d,mn)}} \circ T_{\frac{mn}{d'(d,mn)}}. \qedhere
    \end{equation}
\end{proof}

\section{Trace Formulas} \label{sec: trace formulas}

The Eichler-Selberg trace formula \cites{Eichler, Selberg} expresses the traces of Hecke operators on spaces of modular forms in terms of weighted sums of certain Hurwitz-Kronecker class numbers.

Although this trace formula was later generalized by Arthur \cite{Arthur} to treat more general automorphic forms, we narrow our current discussion to classical modular forms and refer the reader to \cite{Knapp} for an excellent introduction.

\subsection{Trace Formula for {$\SL_2(\Z)$}}
Before stating explicit expression for the traces of $T_p \circ W_N$ on the spaces of cusp forms and their new subspaces, we need to introduce some definitions and notations.

A positive definite \defi{binary quadratic form} $Q(x,y) = ax^2 + bxy + cy^2 \in \Z[x,y]$ with integral coefficients $a,b,c \in \Z$ and \defi{discriminant} $n = 4ac-b^2$ is such that $n > 0$ and $a > 0$. 
Let $B_n$ be the space of positive definite binary quadratic forms of discriminant $n$.
$\SL_2(\Z)$ acts naturally on $B_n$, and we let
$$
H(n) = 2 \frac{| \SL_2(\Z) \backslash B_n |}{| \Stab_{\SL_2(\Z)} (Q) |}
$$
be the number of equivalence classes divided by the number of automorphisms of any form $Q \in B_n$. We extend $H$ to all integers, by setting $H(0) = -\frac{1}{12}$ and $H(n) = 0$ if $n < 0$.

\begin{exa}
    If $n = -3$, there is a unique equivalence class of quadratic forms of discriminant $-3$, namely $Q(x,y) = x^2 + xy + y^2$. We have
    $$
    \Stab_{\SL_2(\Z)} (Q) = \pm \left \{ \mat 1001, 
    \mat 11{-1}0, \mat 01{-1}{-1}
    \right \},
    $$
    reflected in $Q(x,y) = Q(x+y,-x) = Q(y, -x-y)$.
    Therefore, $H(3) = \frac{1}{3}$.
\end{exa}


For every even $k > 0$, we define a polynomial $p_k(t,N) \in \Z[t,N]$ as the coefficient of $x^{k-2}$ in the power series development of $(1-tx+Nx^2)^{-1}$. 

\begin{exa}
    We have $p_2(t,N) = 1$ and $p_4(t,N) = t^2 - N$.
\end{exa}

\begin{remark}
    The appearance of the polynomials $p_k(t,N)$ might seem surprising at first glance. However, note that if $M \in M_2(\Z)$ is a matrix with $\Tr(M) = t$ and $\det (M) = N$, and we consider the symmetric representation $\rho_{k-2} : M_2(\Z) \to \Sym^{k-2} \Z^2$, then $\Tr(\rho_{k-2}(M)) = p_k(t,N)$.This makes sense recalling the Eichler-Shimura isomorphism.
\end{remark}

Perhaps in its simplest form, the trace formula for $\SL_2(\Z)$ can be stated as follows.
\begin{thm} \cite{Zagier} Let $k \ge 4$ be an even integer, and let $n > 0$ be an integer.
    \begin{equation} \label{eq: Eichler-Selberg}
    -2 \Tr (T_n | S_k(\SL_2(\Z))) = \sum_{t \in \Z} p_k(t,n) H(4n-t^2) 
    + \sum_{d \mid n} \min \left(d, \frac{n}{d} \right)^{k-1}
    \end{equation}
\end{thm}

\begin{exa}
    When $k = 4$, there are no cusp forms of weight $4$, hence $S_4(\SL_2(\Z)) = 0$ and the trace would vanish for any $n$. For example, if $n = 5$, we obtain
    \begin{align*}
    \sum_{t \in \Z} p_4(t,5) H(20-t^2) &= -5H(20) - 8H(19) - 2H(16) + 8H(11) + 22H(4) \\
    &= -10 - 8 - 3 + 8 + 11 = -2,
    \end{align*}
    precisely cancelling the contribution from the last term.
\end{exa}

\subsection{Trace formula for higher levels}

Following \cite{P}, we introduce a trace formula for the action of the operator $T_n \circ W_Q$ on the space $S_k(N)$. 

But first, we have to introduce some more notations.

Let $N, t, n, u \in \Z$ be such that $u \mid N$ and $u^2 \mid t^2 - 4n$. Denote
\begin{equation} \label{eq: S_N}
S_N(u,t,n) = \{ \alpha \in (\Z / N \Z)^{\times} : \alpha^2 - t \alpha + n \equiv 0 \pmod {Nu} \},
\end{equation}
and let $\varphi_1(N) = [\SL_2(\Z) : \Gamma_0(N)] = N \prod_{p \mid N} \left(1 + \frac{1}{p} \right) $.
Write 
$$
B_N(u,t,n) = \frac{\varphi_1(N)}{\varphi_1(N/u)} |S_N(u,t,n)|, 
$$
and 
\begin{equation} \label{eq: C_N}
C_N(u,t,n) = \sum_{d \mid u} B_N(u/d, t, n) \mu (d)
\end{equation}
for its M{\"o}bius inverse.

Let $N,Q,a,d$ be integers such that $Q \parallel N$. Define
$$
\Phi_{N,Q}(a,d) = \frac{\varphi(Q)}{Q} \sum_{\substack{\frac{N}{Q} = rs, (r,s) \mid a-d \\ (r,a) = 1, (s,d) = 1}} \varphi((r,s)).
$$

Finally, if $n,N$ are integers, denote
$$
\sigma_{1,N}(n) = \sum_{d \mid n, (N,d) = 1} \frac{n}{d}.
$$

We are now ready to write down the trace formula.

\begin{thm}[\cite{P}*{Theorem 4}] \label{thm: Popa trace formula}
Let $k \ge 2$ be even. Then
\begin{align}
\nonumber
\Tr(T_n \circ W_Q | S_k(N)) = 
&-\frac{1}{2} \sum_{\substack{t^2 \le 4Qn \\ Q \mid t }}
\frac{p_k(t,Qn)}{Q^{\frac{k}{2} - 1}} 
\sum_{\substack{u \mid Q \\ u' \mid \frac{N}{Q}}} H \left( \frac{4Qn - t^2}{(uu')^2} \right) C_{\frac{N}{Q}} (u', t, Qn) \mu (u) \\
&- \frac{1}{2} \sum_{\substack{Qn=ad \\ Q \mid a+d}} 
\frac{\min(a,d)^{k-1}}{Q^{\frac{k}{2}-1}} \cdot \Phi_{N, Q}(a,d) + \delta_{k,2} \sigma_{1,N}(n)
\end{align}
\end{thm}

\begin{remark}
    Implementing this function according to \cite{P} revealed a slight error in \cite{P}*{Lemma 4.5}. This error is corrected in Appendix~\ref{sec: fixing SZ}.
\end{remark}

\begin{exa}
When we substitute $Q = N = 1$ and $k \ge 4$, we obtain
$$
-2 \Tr(T_n | S_k(1)) = \sum_{t^2 \le 4n} p_k(t,n) H(4n-t^2) 
+ \sum_{d \mid n} \min \left(d,\frac{n}{d} \right)^{k-1},
$$
where we have used 
$$
C_1(1,t,n) = B_1(1,t,n) = |S_1(1,t,n)| = 1,
$$
and $\Phi_{1,1}(a,d) = 1$, recovering \eqref{eq: Eichler-Selberg}.
\end{exa}

Specializing to the case where $n = 1$ and $Q = N$, we obtain the following.

\begin{cor} \label{cor: Popa al trace}
    Let $k \ge 2$ be even. Then
\begin{align}
\nonumber
\Tr(W_N |S_k(N)) = 
&-\frac{1}{2} \sum_{\substack{t^2 \le 4N \\ N \mid t }}
\frac{p_k(t,N)}{N^{\frac{k}{2} - 1}} 
\sum_{u \mid N} H \left( \frac{4N - t^2}{u^2} \right) \mu (u) \\
&- \frac{\varphi(N)}{2N} \sum_{\substack{N=ad \\ N \mid a+d}} 
\frac{\min(a,d)^{k-1}}{N^{\frac{k}{2}-1}}  + \delta_{k,2}
\end{align}
\end{cor}

\begin{remark}
    If $N > 4$, this simplifies greatly to be
    \begin{equation} \label{eq: W_N N > 4}
    \Tr(W_N |  S_k(N)) = \frac{(-1)^{\frac{k}{2}}}{2} 
    \sum_{u \mid N} H \left( \frac{4N}{u^2} \right) \mu(u) + \delta_{k = 2}.
    \end{equation}
\end{remark}

\subsection{Trace formula for newforms}

Although Theorem~\ref{thm: Popa trace formula} is already quite useful, we would like to obtain a formula for the trace on the new subspace $S_k(N)^{\new}$. In order to do that, we will express $\Tr(T_n \circ W_Q | S_k(N))$ as a linear combination of $\Tr(T_{n'} \circ W_{Q'} | S_k(N')^{\new})$ for some $n',Q',N'$ which are smaller than $n, Q, N$. Then, using M{\"o}bius inversion, we will obtain a formula for the trace on $S_k(N)^{\new}$.

For $N' \mid N$, let us denote
\begin{equation}\label{eq: image of new subspace}
    \iota_{N,N'} \left (S_k(N')^{\new} \right) \colonequals 
    \bigoplus_{d \mid \frac{N'}{N}} 
    \alpha_d \left( S_k(N')^{\new} \right).
\end{equation}

For integers $d \mid n' \mid n$ and an integer $N$, let us define the sets
\begin{equation} \label{eq: divisors primr to Nprime}
        \calN_{N}(d) = \left \{ 
        N' \mid N : 
        d \mid \frac{N}{N'}, \quad (d, N') = 1, 
        \right \},
\end{equation}
and
\begin{equation} \label{eq: square divisors}
        \calN_{N, n}(d,n') = \left \{ 
        N' \in \calN_{N}(d) : 
        \frac{N}{N' n'} = \square,  
        \left( 
        \frac{N}{N'}, dn
        \right) = n'
        \right \}.
\end{equation}

\begin{lem}
    Let $n, N$ be integers, and let $Q \parallel N$. Write $Q' = (Q,N')$, and for any integer $d$ write $d_Q = (d, Q)$ and $d_{N/Q} = \frac{d}{d_Q}$ so that $d = d_Q d_{N/Q}$. If $d \mid N$, then it further holds that $d_{N/Q} = (d, N/Q)$ and $(d_Q, d_{N/Q}) = 1$.
    For integers $d \mid n' \mid n$, we define the set
    \begin{equation} \label{eq: N(d,n) sets}
    \calN_{Q,N,n}(d, n') =  \calN_{Q,n_Q}(d_Q, n_Q') \cdot \calN_{N/Q}(d_{N/Q}).
    \end{equation}
    Then
    \begin{align*}
        \Tr&(T_n \circ W_Q | S_k(N)) \\
        &= \sum_{\substack{d | n' \mid n \\ n'_{N/Q} = d_{N/Q}^2}}
        (n')^{\frac{k}{2}} 
        \frac{\mu(d)}{d}
        \sum_{N'}
        \sigma_{0,n} \left( 
        \frac{N/Q}{N'/Q'}
        \right)
        \Tr \left(T_{\frac{n}{n'}} \circ W_{Q'} \middle | S_k(N')^{\new} \right).
    \end{align*} 
    where $N'$ ranges over the set $\calN_{Q,N,n}(d, n')$.
\end{lem}
\begin{proof}
    By \cite{DS}*{Theorem 5.8.2}, the space $S_k(N')^{\new}$ admits a basis of newforms $\{f_1, \ldots, f_r \}$. Therefore, the forms $\{ \alpha_D(f_j) : D \mid \frac{N}{N'}, 1 \le j \le r \}$ constitute a basis for the space $\iota_{N,N'}(S_k(N')^{\new})$. Since each $f_j$ is an eigenform for the Hecke operators, there exist $a_{m,j}$ such that $T_m(f_j) = a_{m,j} f_j$ for all integers $m$ and all $j = 1,2,\ldots,r$.
    
    Since $W_{Q'}$ commutes with the $T_m$, the $f_j$ are also eigenforms for $W_{Q'}$, hence there are $\varepsilon_j \in \{-1,1\}$ such that $W_{Q'}(f_j) = \varepsilon_j f_j$.

    Using Lemma~\ref{lem: AL and level-raising},
    we compute for any $D \mid \frac{N}{N'}$ that
    \begin{align} \label{eq: T_n W_Q action on basis 1}
    \nonumber
    (T_n \circ W_Q)(\alpha_D(f_j)) &= \left( \frac{D}{D'} \right)^{\frac{k}{2}-1} (T_n \circ \alpha_{D'})(W_{Q'} (f_j)) \\
    &= \varepsilon_j\left( \frac{D}{D'} \right)^{\frac{k}{2}-1} (T_n \circ \alpha_{D'} )(f_j),
    \end{align}
    where $D' = \left(\frac{N}{DN'}, Q \right) \cdot \left( D, \frac{N}{Q} \right) = \frac{Q}{Q'D_Q} \cdot D_{N/Q}$. 
    Applying Corollary~\ref{cor: Hecke and level raising general}, we get
    \begin{align}
    \nonumber
    (T_n \circ \alpha_{D'})(f_j) &= (D', n)^{\frac{k}{2}-1}
     \sum_{\substack{d \mid \left( \frac{n}{(D',n)}, \frac{N}{N'} \right) \\ (d, N') = 1}} \mu(d) \cdot \left(\alpha_{\frac{D'd}{(D',n)}} \circ T_{\frac{n}{d(D',n)}} \right)(f_j) \\
     &= (D', n)^{k-1} \sum_{\substack{d \mid \left( \frac{n}{(D',n)}, \frac{N}{N'} \right) \\ (d, N') = 1}} \mu(d) a_{\frac{n}{d(D',n)},j} \cdot \alpha_{\frac{D'd}{(D',n)}}(f_j),
    \end{align}
    which combined with \eqref{eq: T_n W_Q action on basis 1} transforms to
    \begin{equation*}
        (T_n \circ W_Q)(\alpha_D(f_j)) =
        \varepsilon_j
        \left( \frac{D}{D'} \right)^{\frac{k}{2}-1}
        (D', n)^{k-1} \sum_{\substack{d \mid \left( \frac{n}{(D',n)}, \frac{N}{N'} \right) \\ (d, N') = 1}} \mu(d) a_{\frac{n}{d(D',n)},j} \cdot \alpha_{\frac{D'd}{(D',n)}}(f_j).
    \end{equation*}
    Note that the above sum only contributes to the trace if there exists $d \mid \left(\frac{n}{(D',n)}, \frac{N}{N'}\right)$ such that $(d, N') = 1$ and $D'd = D(D',n)$. 

    Assume this is the case, and fix such an integer $d$.
    Since $(d, N') = 1$, we have $(d_Q, Q') = 1$ and $(d_{N/Q}, N'/Q') = 1$, and from $d \mid \frac{N}{N'}$ we deduce that $d_Q \mid \frac{Q}{Q'}$ and $d_{N/Q} \mid \frac{N/Q}{N'/Q'}$, so that $Q' \in \calN_Q(d_Q)$ and $N'/Q' \in \calN_{N/Q}(d_{N/Q})$.
    
    Denote $n' = d(D',n)$, and note that $n' \mid n$. Then $(D',n) = \frac{n'}{d}$, hence 
    $$
    \frac{Q d_Q}{Q'D_Q} \cdot D_{N/Q} d_{N/Q} = D' d = D \frac{n'}{d} = D_Q \frac{n_Q'}{d_Q} \cdot D_{N/Q} \frac{n_{N/Q}'}{d_{N/Q}},
    $$
    which leads to
    \begin{equation} \label{eq: square condition D_Q}
    D_Q^2 = \frac{Q d_Q^2}{Q'n_Q'}, \quad
    d_{N/Q}^2 = n_{N/Q}'.
    \end{equation}
    Since $(D', n) = \frac{n'}{d}$, we have 
    $(D_Q', n_Q) = \frac{n_Q'}{d_Q}$, hence
    \begin{equation} \label{eq: coprime condition D_Q}
    \left(
    \frac{D_Q}{d_Q}, \frac{n_Q d_Q}{n_Q'}
    \right)
    = \left(
    \frac{Qd_Q}{Q'D_Qn_Q'}, \frac{n_Q d_Q}{n_Q'}
    \right)
    = \left( \frac{D_Q'd_Q}{n_Q'}, \frac{n_Q d_Q}{n_Q'} \right) = 1,
    \end{equation}
    with all the above quantities being integral. 
    
    In particular, as $\frac{Q}{Q'n_Q'} = \frac{D_Q}{d_Q} \cdot \frac{Qd_Q}{Q'D_Qn_Q'}$, we see that $n_Q' \mid \frac{Q}{Q'}$ and $\left(\frac{Q}{Q'n_Q'}, \frac{d_Q n_Q}{n_Q'} \right) = 1 $, hence $\left(\frac{Q}{Q'}, d_Q n_Q \right) = n_Q'$.
    Returning to \eqref{eq: square condition D_Q}, we see that $\frac{Q}{Q' n_Q'} = \left(\frac{D_Q}{d_Q} \right)^2$ is a perfect square, showing that $Q' \in \calN_{Q,n_Q}(d_Q, n_Q')$, hence $N' \in \calN_{Q,N,n}(d,n')$ for some $n' \mid n$ with $n_{N/Q}' = d_{N/Q}^2$ and $d \mid n'$.

    Conversely, if $N' \in \calN_{Q,N,n}(d,n')$ for some $n' \mid n$ with $n_{N/Q}' = d_{N/Q}^2$ and $d \mid n'$, then $\frac{Q}{Q'n_Q'}$ is a square, showing the existence of a unique integer $D_Q$ such as in \eqref{eq: square condition D_Q}. 
    Let $D_{N/Q} \mid \frac{N/Q}{N'/Q'}$ be such that $(D_{N/Q}, n) = d_{N/Q}$, and consider $D = D_Q D_{N/Q}$.

    Since $\left( \frac{Q}{Q'}, d_Q n_Q \right) = n_Q'$, $n_Q' \mid \frac{Q}{Q'}$ showing that $d_Q \mid D_Q$, and \eqref{eq: coprime condition D_Q} follows, showing that $(D_Q', n_Q) = \frac{n_Q'}{d_Q}$.
    By construction, $(D_{N/Q}',n) = (D_{N/Q}, n) = d_{N/Q}$, hence $(D', n) = \frac{n_Q'}{d_Q} \cdot d_{N/Q} = \frac{n'}{d}$. From \eqref{eq: square condition D_Q} we obtain $D'd = D(D',n)$, and as $d(D',n) = n' \mid n$, $d \mid \frac{N}{N'}$ and $(d,N') = 1$, we see that $\alpha_D(f_j)$ contributes to the trace.
    Moreover, in such a case, we have
    $$
    \frac{D}{D'} = \frac{d_Q^2}{n_Q'} = \frac{d^2}{n'}, \quad
    (D', n) = \frac{n'}{d}.
    $$
    Therefore, 
    \begin{align*}
    \Tr(T_n \circ W_Q | & \iota_{N,N'}(S_k(N')^{\new})) \\
    &= \sum _j \sum_{\substack{d \mid n' \mid n \\ d_{N/Q}^2 = n_{N/Q}' \\ N' \in \calN_{Q,N,n}(d,n')}} 
    \sum_{(D_{N/Q}, n) = d_{N/Q}}
    \varepsilon_j \left( \frac{d^2}{n'} \right)^{\frac{k}{2}-1}
        \left(\frac{n'}{d} \right)^{k-1} \mu(d) a_{\frac{n}{n'},j} \\
    &= \sum_{\substack{d \mid n' \mid n \\ d_{N/Q}^2 = n_{N/Q}' \\ N' \in \calN_{Q,N,n}(d,n')}} 
    \sigma_{0,n} \left( \frac{N/Q}{N'/Q'} \right)
    (n')^{\frac{k}{2}} \frac{\mu(d)}{d} 
    \sum_j
    \varepsilon_j a_{\frac{n}{n'},j}.
    \end{align*}

    Recall (e.g. \cite{DS}*{Section 5.7}) that we have the following decomposition
    \begin{equation} \label{eq: newspace decomposition}
    S_k(N) = \bigoplus_{N' \mid N} 
    \iota_{N,N'} \left (S_k(N')^{\new} \right).
    \end{equation}
    Summing over all $N' \mid N$, we obtain
    $$
    \Tr(T_n \circ W_Q | S_k(N))
    = \sum_{\substack{d \mid n' \mid n \\ d_{N/Q}^2 = n_{N/Q}'}} 
    \sum_{N'}
    \sigma_{0,n} \left( \frac{N/Q}{N'/Q'} \right)
    (n')^{\frac{k}{2}} \frac{\mu(d)}{d} 
    \sum_j
    \varepsilon_j a_{\frac{n}{n'},j},
    $$
    where $N'$ ranges over the set $\calN_{Q,N,n}(d,n')$,
    which proves the claim.
\end{proof}

In order to perform M{\"o}bius inversion, we would need to separate out the terms where $n' = 1$. This is the content of the next corollary.

\begin{cor}
    Let $n, N$ be positive integers, and let $Q \parallel N$. 
    Denote 
    \begin{align*}
    T_{<n,k} &(Q,N) \\
    &= \sum_{\substack{d | n' \mid n \\ n'_{N/Q} = d_{N/Q}^2 \\ 1 < n'}}
        (n')^{\frac{k}{2}} 
        \frac{\mu(d)}{d}
        \sum_{N'}
        \sigma_{0,n} \left( 
        \frac{N/Q}{N'/Q'}
        \right)
        \Tr \left(T_{\frac{n}{n'}} \circ W_{Q'} \middle | S_k(N')^{\new} \right),
    \end{align*}
    where $N'$ ranges over the set $\calN_{Q,N,n}(d, n')$ defined in \eqref{eq: N(d,n) sets}.
    Then
    \begin{align} \label{eq: formula before inversion}
        \Tr&(T_n \circ W_Q | S_k(N)) \\
        \nonumber
        &= \sum_{\substack{N' \mid N \\ \frac{Q}{Q'} = \square \\ \left( \frac{Q}{Q'}, n \right) = 1}}
        \sigma_{0,n} \left( 
        \frac{N/Q}{N'/Q'}
        \right)
        \Tr \left(T_n \circ W_{Q'} \middle | S_k(N')^{\new} \right) + T_{<n,k}(Q,N).
    \end{align}    
\end{cor}

\section{Arithmetic functions and M{\"o}bius inversion}

The final step is using M{\"o}bius inversion to obtain a formula for the trace on the newspace from the traces on the full spaces of modular forms of lower levels.
In order to do that, we need to establish some properties of the functions being used. 

Recall that a function $f : \N \to \C$ is called an \defi{arithmetic function}, and the arithmetic functions form a commutative ring with respect to \defi{Dirichlet convolution}
$$
(f * g)(N) = \sum_{N' \mid N} f(N') g\left(\frac{N}{N'} \right),
$$
with multiplicative unit $\delta_1$ given by 
$$
\delta_1(N) = \begin{cases}
    1 & N = 1 \\
    0 & N \ne 1
\end{cases}.
$$

Since $Q$ depends on $N$, we require some changes in order to
write our formulas as identities in the Dirichlet ring.

Let $\calQ$ be a fixed set of prime numbers, and for a number $N$, let $Q_{\calQ}(N) \mid N$ be the largest divisor of $N$ divisible only by primes in $\calQ$. 
Note that $Q_{\calQ}(N) \parallel N$ and for any $Q \parallel N$, if $\calQ$ is the set of primes dividing $Q$, then $Q = Q_{\calQ}(N)$.

Fix $k$, $n$ and $\calQ$, and let us write 
$$
F_{\calQ,n,k}(N) = \Tr(T_n \circ W_{Q_{\calQ}(N)} | S_k(N)) - T_{<n,k}(N,Q_{\calQ}(N)).
$$

Then $F_{\calQ,n,k}$ is an arithmetic function, 
and similarly we can define the following arithmetic functions.
\begin{align}
\nonumber
G_{\calQ,n,k}(N) &= \Tr \left(T_n \circ W_{Q_{\calQ}(N)} \middle | S_k(N)^{\new} \right), \\
\label{eq: arithmetic multiplicative function}
H_{\calQ, n}(N) &= \begin{cases}
    \sigma_{0,n} \left(\frac{N}{Q_{\calQ}(N)} \right) & (Q_{\calQ}(N), n) = 1, \ Q_{\calQ}(N) = \square \\
    0 & \text{ else }
\end{cases}.
\end{align}

Using these notations, \eqref{eq: formula before inversion} reads
$$
F_{\calQ,n,k} = G_{\calQ,n,k} * H_{\calQ, n}.
$$

The point of M{\"o}bius inversion is that arithmetic functions $f : \N \to \C$ with $f(1) \ne 0$ are invertible. In our case, $H_{\calQ, n}(1) = 1 \ne 0$, hence $H_{\calQ, n}$ is invertible and we can write
\begin{equation} \label{eq: mobius inversion}
G_{\calQ,n,k} = F_{\calQ,n,k} * H_{\calQ, n}^{-1}.
\end{equation}

Note that if $(M,N) = 1$, then $Q_{\calQ}(MN) = Q_{\calQ}(M) Q_{\calQ}(N)$, and $Q_{\calQ}(MN)$ is a square if and only if both $Q_{\calQ}(M)$ and $Q_{\calQ}(N)$ are squares. Since $N/Q_{\calQ}(N)$ and $M/Q_{\calQ}(M)$ are coprime, we also have
$$
\sigma_{0,n} \left( \frac{MN}{Q_{\calQ}(MN)} \right)
=
\sigma_{0,n} \left( \frac{M}{Q_{\calQ}(M)} \right)
\sigma_{0,n} \left( \frac{N}{Q_{\calQ}(N)} \right),
$$
showing that $H_{\calQ,n}$ is multiplicative. 
It follows that $H_{\calQ,n}^{-1}$ is also multiplicative, hence it suffices to describe its values on prime powers. 

\begin{lem}
Let $H_{\calQ, n}$ be the function described in \eqref{eq: arithmetic multiplicative function}.
Let values of $\alpha_{\calQ,n}$ be the multiplicative function whose values on prime powers are described by 
    \begin{equation} \label{eq: alpha}
    \alpha_{\calQ,n}(p^e) = \begin{cases}
        1 & p \nmid n, \quad p \notin \calQ, \quad e = 2 \\
        -2 & p \nmid n, \quad p\notin \calQ, \quad e = 1 \\
        -1 & p \nmid n, \quad p \in \calQ, \quad e = 2 \\
        -1 & p \mid n, \quad p \notin \calQ, \quad e = 1 \\
        0 & else
    \end{cases}.
\end{equation}
Then $\alpha_{\calQ,n} = H_{\calQ, n}^{-1}$ is its inverse with respect to Dirichlet convolution.
\end{lem}

\begin{proof}
    It suffices to show that $\alpha_{\calQ,n} * H_{\calQ,n} = 0$ for every prime power $p^e$ such that $e \ge 1$. We divide to several cases.

    If $p \notin \calQ$, then $Q_{\calQ}(p^i) = 1$ for all $i$, so that 
    $$
    H_{\calQ,n}(p^i) = \sigma_{0,n}(p^i)
    = \begin{cases}
        i + 1 & p \nmid n \\
        1 & p \mid n
    \end{cases}.
    $$
    It follows that 
    \begin{align*}
    (\alpha_{\calQ,n}*H_{\calQ,n})(p^e)
    &= \sum_{i=0}^e \alpha_{\calQ,n}(p^i) H_{\calQ,n}(p^{e-i}) \\
    &= \begin{cases}
        1 \cdot (e+1) - 2 \cdot e + 1 \cdot (e-1) & p \nmid n\\
        1 \cdot 1 - 1 \cdot 1& p \mid n
    \end{cases} = 0.
    \end{align*}

    If $p \in \calQ$, then $Q_{\calQ}(p^i) = p^i$ for all $i$, so that
    $$
    H_{\calQ,n}(p^i) = \begin{cases}
        1  & p \nmid n, \quad 2 \mid i \\
        0 & p \mid n \ \text{ or } \ 2 \nmid i
    \end{cases}.
    $$
    It follows that
    \begin{align*}
    (\alpha_{\calQ,n}*H_{\calQ,n})(p^e)
    &= \sum_{i=0}^e \alpha_{\calQ,n}(p^i) H_{\calQ,n}(p^{e-i}) \\
    &= \begin{cases}
         1 \cdot \delta_{2 \mid e} - 1 \cdot \delta_{2 \mid e} & p \nmid n\\
         0 & p \mid n
    \end{cases} = 0.
    \end{align*}

    As both $\alpha_{\calQ,n}$ and $H_{\calQ,n}$ are multiplicative, it follows that $\alpha_{\calQ,n} * H_{\calQ,n} = \delta_1$.
\end{proof}

Expanding \eqref{eq: mobius inversion}, we obtain our final formula. 

\begin{cor} \label{cor: main thm}
    Let $n, N$ be positive integers, and let $Q \parallel N$.
    Let $\calQ$ be the set of primes dividing $Q$.
    Then
    \begin{align}
    \nonumber
        \Tr&(T_n \circ W_Q | S_k(N)^{\new}) \\ \label{eq: final formula}
        &= \sum_{N' \mid N}
        \alpha_{\calQ,n} \left( \frac{N}{N'} \right)
        \left(
        \Tr \left(T_n \circ W_{Q'} \middle | S_k(N') \right)
        - T_{<n,k}(Q',N')
        \right).
    \end{align}      
\end{cor}

When $n = 1$, $T_{<n,k}(Q,N) = 0$, and we simply obtain
\begin{cor} \label{cor: W_Q new}
Let $N$ be a positive integer, and let $Q \parallel N$. Let $\calQ$ be the set of primes dividing $Q$.
    Then
    \begin{equation}
        \Tr(W_Q | S_k(N)^{\new}) 
        = \sum_{N' \mid N}
        \alpha_{\calQ,n} \left( \frac{N}{N'} \right)
        \Tr \left(W_{Q'} \middle | S_k(N') \right)
        .
    \end{equation}  
\end{cor}

When $Q = N$, we also have $\alpha_{\calQ,1} (m) = \mu(\sqrt{m})$. Combining it with \eqref{eq: W_N N > 4} we obtain the following formula.
\begin{cor}
    Let $N$ and $k > 2$ be positive integers. 
    Assume further that if $d < 4$ divides $N$ then $N/d$ is not a square.
    Then
    $$
    \Tr(W_N | S_k(N)^{\new})
    = \frac{(-1)^{\frac{k}{2}}}{2}  \sum_{\substack{N' | N \\ \frac{N}{N'} = \square}}
    \mu \left(\sqrt{\frac{N}{N'}} \right)
    \sum_{u \mid N'} H \left( \frac{4N'}{u^2} \right) \mu(u) .
$$
\end{cor}

In the special case when $n$ is a prime, we can also write things more explicitly. Indeed, assume $n = p$ is a prime. Then the only divisor $1 < n'$ of $n$ is $n' = p$. 
If $p \nmid Q$, then $n'_{N/Q} = p$ is not a square, so that $T_{<p,k},(Q,N) = 0$.
If $p \mid Q$, then $n'_{N/Q} = 1$, so $d \in \{1, p\}$.
Therefore, if $p \mid Q$, we get
\begin{align} \label{eq: n = p}
T_{<p,k}(Q,N) 
&= p^{\frac{k}{2}} \sum_{N' \in \calN_{Q,N,p}(1,p)} 
\sigma_0 \left(\frac{N/Q}{N'/Q'} \right) \Tr \left( W_{Q'} \middle | 
S_k(N')^{\new}
\right) \\
\nonumber
&- p^{\frac{k}{2}-1} \sum_{N' \in \calN_{Q,N,p}(p,p)} 
\sigma_0 \left( \frac{N/Q}{N'/Q'} \right) \Tr \left( W_{Q'} \middle | 
S_k(N')^{\new}
\right).
\end{align}

We note that $\calN_{Q,N,p}(1,p) = \calN_{Q,p}(1,p) \cdot \calN_{N/Q}(1)$, so that
\begin{equation}
\calN_{Q,N,p}(1,p) = \left \{ N' \mid N : \frac{Q}{Q'p} = \square, \quad p \mid \frac{Q}{Q'} \right \},
\end{equation}
and 
$\calN_{Q,N,p}(p,p) = \calN_{Q,p}(p,p) \cdot \calN_{N/Q}(1)$, so that 
\begin{equation}
\calN_{Q,N,p}(p,p) = \left \{ N' \mid N : \frac{Q}{Q'p} = \square, \quad p \parallel \frac{Q}{Q'}, \quad p \nmid Q' \right \}.
\end{equation}

In particular, if $p^2 \mid Q$, then $\calN_{Q,N,n}(p,p) = \emptyset$.

If we further specialize to the case $Q = N$, we obtain
\begin{align}
\nonumber
T_{<p,k}& (N,N) \\ \label{eq: n = p, Q = N}
&= p^{\frac{k}{2}} \sum_{N' \in \calM_N(1)} \Tr(W_{N'} | S_k(N')^{\new})
- p^{\frac{k}{2}-1} \sum_{N' \in \calM_N(p)} \Tr(W_{N'} | S_k(N')^{\new}),
\end{align}
where $\calM_N(d) = \calN_{N,N,p}(d,p)$ for $d \in \{1,p \}$.

Before continuing, we note the following identities for any  function $G$.
Write 
$$
s_d(G) = \sum_{\substack{N' | N \\ p \nmid \frac{N}{N'}}} \mu \left( \sqrt{\frac{N}{N'}} \right) \sum_{N'' \in \calM_{N'}(d)} G(N'')
$$
for $d \in \{1,p\}$.
Then
\begin{equation*}
s_1(G) = \sum_{\substack{N'' | N \\ \frac{N}{N''p} = \square}}
G(N'') \sum_{\substack{N'' | N' | N \\ p \nmid \frac{N}{N'} = \square}} \mu \left( \sqrt{\frac{N}{N'}} \right) 
= \sum_{i=0}^{\left[\frac{v_p(N)-1}{2}\right]} G \left( \frac{N}{p^{2i+1}} \right),
\end{equation*}

and
\begin{equation*}
s_p(G)
= \sum_{\substack{N'' | N \\ \frac{N}{N''p} = \square}}
G(N'') \sum_{\substack{pN'' | N' | N \\ p \nmid \frac{N}{N'} = \square, p^2 \nmid N'}} \mu \left( \sqrt{\frac{N}{N'}} \right) 
= \delta_{p^2 \nmid N} \cdot G \left( \frac{N}{p} \right).
\end{equation*}

Plugging \eqref{eq: n = p, Q = N} back into \eqref{eq: final formula}, and using the above identities with $G = G_{\calQ, p, k}$ we obtain the following corollary.

\begin{cor} \label{cor: T_p W_N new}
    Let $p$ be a prime, let $N,k$ be positive integers with $k \ge 2$ even. 
    Let $w_k(N) = \Tr(W_N | S_k(N)^{\new})$, and if $p \mid N$ let
    $$
    W_{<N,p} = \delta_{p^2 \nmid N} \cdot p^{\frac{k}{2}-1} w_k \left(\frac{N}{p} \right)
     - p^{\frac{k}{2}} \sum_{i = 0}^{\left[\frac{v_p(N)-1}{2}\right]} w_k \left(\frac{N}{p^{2i+1}} \right),
    $$
    while if $p \nmid N$, let $W_{<N,p} = 0$.
    Then the following formula holds. 
    \begin{equation*}
    \Tr(T_p \circ W_N | S_k(N)^{\new}) 
    = \sum_{\substack{N' \mid N \\ p \nmid \frac{N}{N'} = \square}}
    \mu \left( \sqrt{\frac{N}{N'}} \right) 
    \Tr(T_p \circ W_{N'} | S_k(N')) + W_{<N,p}.
\end{equation*}
\end{cor}

Note that substituting Theorem~\ref{thm: Popa trace formula} and Corollary~\ref{cor: W_Q new} in the above corollary yields an explicit formula for the trace of $T_p \circ W_N$ on the new subspace.

\appendix

\section{Correction of previous formulas} \label{sec: fixing SZ}

\subsection{Trace formula in \cite{SZ}}
It seems that in \cite{SZ} there is an error in the formula for the trace of the operators on the total cuspidal subspace. Let us write down the correct version of this formula (second formula, p. 117, in the block of 3 formulas). 

\begin{thm} \label{thm: SZ fix}
Assume $(l,m) = 1$ and $n \mid m$ with $\left(n, \frac{m}{n} \right) = 1$. Let $T_{l}$ and $W_n$ denote the $l$-th Hecke operator and $n$-th Atkin-Lehner involution on $M_{2k-2}(m)$ as in \cite{SZ}. Then
    $$
    \Tr(T_{l} \circ W_n, S_{2k-2}(m)) = 
    \sum_{\substack{m' \mid m \\ m/m' \text{ squarefree}}}
    \mu \left( \frac{n}{(n,m')} \right) s_{k,m'}(l,(n,m')),
    $$
    where $s_{k,m}(l,n)$ is the quantity described in \cite{SZ}*{Theorem 1}.
\end{thm}

\begin{proof}
    Our starting point is the equation \cite{SZ}*{\S 2.5} on page 133. Its derivation is explained in the half page above it, and can also be deduced from the rest of this paper, and is a standard application of Atkin-Lehner-Li theory. We recall the formula here.
    \begin{equation} \label{eq: SZ 5}
        \Tr(T_{l} \circ W_{n_1} | S_{2k-2}(n_1 n_2)) = 
        \sum_{\substack{a_1 \mid n_1 \\ n_1 / a_1 = \square}} \sum_{a_2 b_2 \mid n_2} \Tr(T_{l} \circ W_{a_1} | S_{2k-2} (a_1 a_2)^{\new}).
    \end{equation}
    We combine this formula with the formula obtained in \cite{SZ} for the trace on the new subspace, appearing right after the formula with the error in p. 117. Again, for convenience we recall it here.
    \begin{equation} \label{eq: SZ new}
        \Tr(T_{l} \circ W_n | S_{2k-2}(m)^{\new}) = 
        \sum_{m' \mid m} \alpha \left( \frac{m}{m'} \right) s_{k,m'} (l, (n,m')),
    \end{equation}
    where $\alpha(m)$ is the multiplicative arithmetic function such that $\alpha(p) = \alpha(p^2) = -1$, $\alpha(p^3) = 1$ and $\alpha(p^s) = 0$ for all $s \ge 4$.

    Combining equations \eqref{eq: SZ 5} and \eqref{eq: SZ new}, we get
    \begin{equation} \label{eq: SZ 5 + new}
        \Tr(T_{l} \circ W_n | S_{2k-2}(m))
        = \sum_{\substack{n' \mid n \\ n / n' = \square}} 
        \sum_{a_2 b_2 \mid \frac{m}{n}} 
        \sum_{m' \mid n' a_2} \alpha \left( \frac{n' a_2}{m'} 
        \right) s_{k,m'}(l, (n', m')).
    \end{equation}
    We note that $b_2$ does not modify the sums, hence summing over the possible $b_2$, we only count the number of divisors of $\frac{m}{n a_2}$. Write $\sigma_0(n)$ for the number of divisors of $n$, as in \cite{SZ}. In addition, since $\left(n, \frac{m}{n} \right) = 1$, we have also $(n', a_2) = 1$, so the divisors $m' \mid n' a_2$ are in bijection with pairs $(m_1, m_2)$ such that $m_1 \mid n'$ and $m_2 \mid a_2$. Moreover, $\alpha$ is multiplicative on coprime inputs, hence rewriting \eqref{eq: SZ 5 + new} we have
    \begin{align} \label{eq: SZ big sum}
    \nonumber
        &\Tr(T_{l} \circ W_n | S_{2k-2}(m)) \\ 
        &=\sum_{\substack{n' \mid n \\ n / n' = \square}} 
        \sum_{a_2 \mid \frac{m}{n}} \sigma_0 \left( \frac{m}{na_2} \right)
        \sum_{m_1 \mid n'}  \alpha \left( \frac{n'}{m_1} 
        \right) \sum_{m_2 \mid a_2} \alpha \left( \frac{ a_2}{m_2} 
        \right) s_{k,m_1 m_2}(l, m_1) \\
        \nonumber
        &= \sum_{m_1 \mid n} 
        \left( \sum_{\substack{m_1 \mid n' \mid n \\ n/n' = \square}} \alpha \left( \frac{n'}{m_1} \right)  \right)
        \sum_{m_2 \mid \frac{m}{n}} \left( 
        \sum_{r \mid a_2 \mid \frac{m}{n}} 
        \sigma_0 \left( \frac{m}{na_2} \right)
        \alpha \left( \frac{a_2}{m_2} \right)
        \right)
        s_{k, m_1 m_2} (l, m_1).
    \end{align}
    In order to simplify this expression, we note that 
    $$
    \sum_{\substack{d \mid n \\ n/d = \square}} \alpha(d) = \mu(n), \quad
    \sum_{d \mid n} \alpha(d) \sigma_0 \left( \frac{n}{d} \right) = \begin{cases}
        1 & n \text{ is squarefree} \\
        0 & \text{else}
    \end{cases},
    $$
    simplifying \eqref{eq: SZ big sum} to 
    \begin{equation}
        \Tr(T_{l} \circ W_n | S_{2k-2}(m)) 
        = \sum_{m_1 \mid n} \mu \left( \frac{n}{m_1} \right)
        \sum_{\substack{m_2 \mid \frac{m}{n} \\ 
        \frac{m}{m_2 n} \text{ squarefree}}} s_{k, m_1 m_2}(l, m_1).
    \end{equation}
    Finally, using again the bijection $m' \leftrightarrow (m_1, m_2)$ (and noting that $\mu$ is zero whenever the argument is not squarefree), we obtain
    $$
    \Tr(T_{l} \circ W_n | S_{2k-2}(m)) 
    = \sum_{\substack{m' \mid m \\ m/m' \text{ squarefree}}} \mu \left( \frac{n}{(n,m')} \right)
    s_{k,m'} (l ,(n,m')),
    $$
    as claimed.
\end{proof}

\subsection{Trace formula in \cite{P}}
While the trace formula appearing in \cite{P}*{Theorem 4} (Theorem~\ref{thm: Popa trace formula}) is correct, there seems to be an error in the formula for computing the quantity $C_N(u,t,n)$, defined in \eqref{eq: C_N}. Let us present the correct version of this formula \cite{P}*{Lemma 4.5}. The only difference is at $p = 2$, specifically when $N=u=2^a$. This error, however, does not impact the validity of the trace formula. Indeed, the difference presented happens specifically when $\frac{t^2 - 4n}{u^2} \bmod 4 \in \{2,3 \}$, hence $\frac{t^2-4n}{u^2}$ is not a discriminant, yielding $H \left(\frac{t^2-4n}{u^2} \right) = 0$ so that the sum in \cite{P}*{Theorem 4} is not affected by it.

\begin{lem} \label{lem: Popa fix}
    For any $N \ge 1$ and for $u \mid N, u^2 \mid t^2 - 4n$ we have 
    $$
    C_N(u,t,n) = |S_N(t,n)| \cdot C_N(u,t^2-4n),
    $$
    where $S_N(t,n) = \{\alpha \in (\Z / N \Z)^{\times} : \alpha^2 - t \alpha + n \equiv 0 \pmod N \}$, and the coefficients $C_N(u,D)$, defined for $u \mid N, u^2 \mid D$ are multiplicative in $(N,u)$, namely
    $$
    C_N(u,D) = \prod_{p \mid N} C_{p^{\nu_p(N)}} (p^{\nu_p(u)},D).
    $$
    If $N = p^a$ with $p \ne 2$ a prime and $a \ge 1$ we have $C_N(p^0, D) = 1, C_N(p^a,D) = p^{\lceil \frac{a}{2} \rceil}$, and setting $b = \nu_p(D)$ (with $b = \infty$ if $D = 0$), for $0 < i < a$ we have:
    \begin{equation} \label{eq: C_N base cases}
    C_N(p^i,D) = \begin{cases}
        p^{\lceil \frac{i}{2} \rceil} - p^{\lceil \frac{i}{2} \rceil - 1} & 1\le i \le b-a, i \equiv a \Mod 2 \\
        -p^{\lceil \frac{i}{2} \rceil - 1} & i = b - a + 1, i \equiv a \Mod 2 \\
        p^{\lfloor \frac{i}{2} \rfloor} \left( \frac{D/p^b}{p} \right) & i = b - a + 1, i \not \equiv a \Mod 2 \\
        0 & \text{ otherwise.}
    \end{cases}
    \end{equation}
    If $N=2^a$ with $a \ge 1$ and $b = \nu_p(D)$, then $C_N(2^0, D) = 1$ and 
    \begin{equation} \label{eq: C_N 2^a}
    C_N(2^a, D) = \begin{cases}
        2^{\lceil \frac{a}{2} \rceil} & b \ge 2a+2 \text{ or } b = 2a, D/2^b \equiv 1 \Mod 4 \\
        -2^{\lceil \frac{a}{2} \rceil - 1} & b = 2a+1 \text{ or } b = 2a, D/2^b \equiv -1 \Mod 4.
    \end{cases}
    \end{equation}
    For $0 < i < a$ we have:
    \begin{equation} \label{eq: C_N 2^i}
    C_N(2^i,D) = \begin{cases}
        2^{\lceil \frac{i}{2} \rceil - 1} & 1\le i \le b-a-2, i \equiv a \Mod 2 \\
        -2^{\lceil \frac{i}{2} \rceil - 1} & i = b - a - 1, i \equiv a \Mod 2 \\
        2^{\lceil \frac{i}{2} \rceil - 1} \epsilon_4 \left(\frac{D}{2^b} \right) & i = b - a, i \equiv a \Mod 2 \\
        2^{\lfloor \frac{i}{2} \rfloor} \left( \frac{D/2^b}{2} \right) & i = b - a + 1, i \not \equiv a \Mod 2, \frac{D}{2^b} \equiv 1 \Mod 4 \\
        0 & \text{ otherwise,}
    \end{cases}
    \end{equation}
    where $\left(\frac{\bullet}{p} \right)$ is the Legendre symbol, and $\epsilon_4$ is the nontrivial character mod $4$.
\end{lem}

\begin{proof}
    The proof of \cite{P}*{Lemma 4.5} works verbatim, showing that $C_N(u,D)$ is multiplicative and that if $N = p^a$ it equals
    \begin{equation} \label{eq: C_N from c_N}
    C_N(p^i,D) = c_N(p^i,D) - \delta_{i \ne 0} \cdot c_N(p^{i-1},D),
    \end{equation}
    where
    \begin{equation} \label{eq: c_N}
    c_N(p^i,D) = \frac{\varphi_1(p^a)/\varphi_1(p^{a-i}) \cdot |S_N(p^i, t, n)|}{|S_N(t,n)|},
    \end{equation}
    and $\varphi_1(N) = [\SL_2(\Z) : \Gamma_0(N)] = |\P^1(\Z / N \Z)| = N \prod_{p \mid N} (1 + 1/p)$. Here, $S_N(u,t,n)$ is defined as in \eqref{eq: S_N}. Note also that $\varphi_1(p^a) / \varphi_1(p^{a-i}) = p^i$ unless $i = a$, when it is $p^a + p^{a-1}$.
    It only remains to compute the cardinality of the set $S_{p^a}(p^i,t,n)$ for $0 \le i \le a$. Since $c_N$ only depends on $D = t^2 - 4n$, we may assume that $p \nmid t$ when $p \ne 2$, and that $4 \nmid t$ when $p = 2$. We may also assume that $S_N(t,n) \ne \emptyset$.

    We have $\alpha^2 - t\alpha+n \equiv 0 \Mod{Nu} $ if and only if $(2\alpha - t)^2 \equiv t^2-4n \Mod{4Nu}$. Writing $D = p^b D_0$ and $Nu = p^{a+i}$, we see that if $a+i \le b$ and $p \ne 2$, it is equivalent to $\alpha \equiv 2^{-1} t \Mod{p^{\lceil \frac{a+i}{2} \rceil}}$, hence $|S_{p^a}(p^i, t, n)| = p^{\lfloor \frac{a-i}{2} \rfloor}$ in this case. 
    When $p = 2$, this still holds if $a+i+2 \le b$, as $2 \mid t$.  
    
    Otherwise, we must have $2\nu_p(2\alpha - t) = b$, so if $2 \nmid b$, $|S_{p^a}(p^i,t,n)| = 0$, while if $2 \mid b$, $(p^{-b/2}(2\alpha-t))^2 \equiv D_0 \Mod{4p^{a+i-b}}$. As $p \nmid D_0$, if $b < a + i$, this equation has $1 + \left( \frac{D_0}{p} \right)$ solutions modulo $4p$, and by Hensel's Lemma, each of them lifts uniquely modulo $4p^{a+i-b}$, determining $\alpha$ modulo $4p^{a+i-b/2}$, hence if $b < a + i$, we have $|S_{p^a}(p^i,t,n)| = p^{\frac{b}{2}-i} \left(1 + \left( \frac{D_0}{p} \right) \right)$. 
    
    Therefore, for $p \ne 2$ we obtain
    $$
    \left |S_{p^a}(p^i, t, n) \right | = \begin{cases}
        p^{\lfloor \frac{a-i}{2} \rfloor} & i \le b - a \\
        0 & 2 \nmid b, b-a+1 \le i \\
        p^{\frac{b}{2}-i}\left(1 + \left( \frac{D/p^b}{p} \right) \right) & 2 \mid b, b-a+1 \le i,
    \end{cases}
    $$
    and using the case $i = 0$ for the denominator in \eqref{eq: c_N}, we get
    $$
    \left |c_{p^a}(p^i, D) \right| = \begin{cases}
        p^{\lceil \frac{a}{2} \rceil} + p^{\lceil \frac{a}{2} \rceil - 1} & i = a \\
        p^{\lfloor \frac{a+i}{2} \rfloor - \lfloor \frac{a}{2} \rfloor}  & i \le b - a, i \ne a \\
        0 & 2 \nmid b, b-a+1 \le i \\
        p^{\frac{b}{2}-\lfloor \frac{a}{2} \rfloor}\left(1 + \left( \frac{D/p^b}{p} \right) \right) & 2 \mid b, b-a+1 \le i, a \le b \\
        1 & 2 \mid b \le a - 1.
    \end{cases}
    $$
    We may now plug it in \eqref{eq: C_N from c_N} to obtain \eqref{eq: C_N base cases} as well as $C_N(p^a, D) = p^{\lceil \frac{a}{2} \rceil}$ for $i = a$, and $C_N(p^0, D) = 1$ for $i = 0$.

    When $p = 2$, one needs to replace $\left( \frac{D_0}{p} \right)$ by the equivalent statement that $D_0$ is a local square, i.e. $D_0 \equiv 1 \Mod{8}$. If $D_0 \equiv 1 \Mod{4}$, then they coincide, so we may write 
    $$|S_{2^a}(2^i, t, n)| = 2^{\frac{b}{2}-i - 1} \left (1 + \epsilon_4(D_0) \right) \left(1 + \left( \frac{D_0}{2} \right) \right).$$
    When $p = 2$ we also have to consider the cases $b = a+i$ and $b = a+i+1$. If $b = a + i$, the number of solutions modulo $4$ is $1 + \epsilon_4(D_0)$, yielding $$|S_{2^a}(2^i, t, n)| = 2^{\frac{b}{2}-i - 1} (1 + \epsilon_4(D_0))$$ in this case, while if $b = a + i + 1$, any $\alpha$ is a solution and we have $$|S_{2^a}(2^i, t, n)| = 2^{\frac{b}{2}-i-1}.$$
    Similarly, for $p = 2$ we obtain
    $$
    \left |S_{2^a}(2^i, t, n) \right | = \begin{cases}
        2^{\lfloor \frac{a-i}{2} \rfloor} & i \le b - a - 2 \\
        0 & 2 \nmid b, b-a-1 \le i \\
        2^{\frac{b}{2}-i-1} \left(1 + \epsilon_4 \left(\frac{D}{2^b} \right) \right) \left(1 + \left( \frac{D/2^b}{2} \right) \right) & 2 \mid b, b-a+1 \le i \\
        2^{\frac{b}{2}-i-1}\left(1 + \epsilon_4 \left(\frac{D}{2^b} \right) \right) & 2 \mid b, i = b - a \\
        2^{\frac{b}{2}-i-1} & 2 \mid b, i = b - a - 1,
    \end{cases}
    $$
    and using the case $i = 0$ for the denominator in \eqref{eq: c_N}, we get
    $$
    \left |c_{2^a}(2^i, D) \right| = \begin{cases}
        2^{\lceil \frac{a}{2} \rceil} + 2^{\lceil \frac{a}{2} \rceil - 1} & i = a, b \ge 2a + 2 \\
        \left( 2^{\lceil \frac{a}{2} \rceil-1} + 2^{\lceil \frac{a}{2} \rceil - 2} \right) \left(1 + \epsilon_4 \left(\frac{D}{2^b} \right) \right) & i = a, b = 2a \\
        2^{\lfloor \frac{a+i}{2} \rfloor - \lfloor \frac{a}{2} \rfloor}  & i \le b - a - 2, i \ne a \\
        0 & 2 \nmid b, b-a-1 \le i \\
        2^{\frac{b}{2}-\lfloor \frac{a}{2} \rfloor - 1} \left(1 + \epsilon_4 \left(\frac{D}{2^b} \right) \right)\left(1 + \left( \frac{D/2^b}{2} \right) \right) & 2 \mid b, 2 \le b-a+1 \le i \\
         1 + \left( \frac{D/2^b}{2} \right) & 2 \mid b, b-a+1 \le i, b=a \\
        2^{\frac{b}{2}-\lfloor \frac{a}{2} \rfloor - 1}\left(1 + \epsilon_4 \left( \frac{D}{2^b} \right) \right) & 2 \mid b, i = b - a, a+1 \le b \\
        2^{\frac{b}{2} - \lfloor \frac{a}{2} \rfloor - 1} & 2 \mid b, i = b - a - 1\\
        1 & 2 \mid b, i = b - a, b = a \\
        1 & 2 \mid b \le a - 1.
    \end{cases}
    $$
    We may now plug it in \eqref{eq: C_N from c_N} to obtain \eqref{eq: C_N 2^i} as well as \eqref{eq: C_N 2^a}.
\end{proof}

\end{document}